\documentclass[a4paper,reqno,oneside,10pt]{amsart}
\usepackage{fullpage}
\usepackage{setspace}
\usepackage{datetime}

\usepackage{amsmath}
\usepackage{amsfonts}
\usepackage{amssymb}
\usepackage{verbatim,enumitem,graphics}
\usepackage{microtype}
\usepackage{xcolor}
\usepackage[backref=page]{hyperref}
\hypersetup{%
    colorlinks,
    linkcolor={red!50!black},
    citecolor={green!50!black},
    urlcolor={blue!50!black}
}
\usepackage{dsfont}

\usepackage[utf8]{inputenc}

\usepackage[abbrev,msc-links,backrefs]{amsrefs} 
\usepackage{doi}

\renewcommand{\PrintDOI}[1]{\doi{#1}}

\usepackage{enumitem}
\usepackage{MnSymbol}

\setenumerate{leftmargin=*} 

\newtheorem{theorem}{Theorem}[section]
\newtheorem{lemma}[theorem]{Lemma}
\newtheorem{proposition}[theorem]{Proposition}
\newtheorem{fact}[theorem]{Fact}

\newtheorem{conjecture}[theorem]{Conjecture}

\newcommand{\oldqed}{}
\def\endofClaim{\hfill\scalebox{.6}{$\Box$}}

\newcommand{\cH}{\mathcal{H}}
\newcommand{\cG}{\mathcal{G}}
\newcommand{\cF}{\mathcal{F}}

\newcommand{\cS}{\mathcal{S}}
\newcommand{\cK}{\mathcal{K}}

\newcommand{\EE}{\mathbb{E}}

\newcommand{\eps}{\varepsilon}
\let\epsilon\varepsilon

\newcommand{\Ktm}{K_t(1,\ldots,1,40)}
\newcommand{\dcup}{\dot\cup}

\renewcommand{\subset}{\subseteq}

\usepackage{lineno}
\newcommand*\patchAmsMathEnvironmentForLineno[1]{%
\expandafter\let\csname old#1\expandafter\endcsname\csname #1\endcsname
\expandafter\let\csname oldend#1\expandafter\endcsname\csname end#1\endcsname
\renewenvironment{#1}%
{\linenomath\csname old#1\endcsname}%
{\csname oldend#1\endcsname\endlinenomath}}%
\newcommand*\patchBothAmsMathEnvironmentsForLineno[1]{%
\patchAmsMathEnvironmentForLineno{#1}%
\patchAmsMathEnvironmentForLineno{#1*}}%
\AtBeginDocument{%
\patchBothAmsMathEnvironmentsForLineno{equation}%
\patchBothAmsMathEnvironmentsForLineno{align}%
\patchBothAmsMathEnvironmentsForLineno{flalign}%
\patchBothAmsMathEnvironmentsForLineno{alignat}%
\patchBothAmsMathEnvironmentsForLineno{gather}%
\patchBothAmsMathEnvironmentsForLineno{multline}%
}

\title{Clique-factors in sparse pseudorandom graphs}

\author[J. Han]{Jie Han}
\author[Y. Kohayakawa]{Yoshiharu Kohayakawa}
\author[P. Morris]{Patrick Morris}
\author[Y. Person]{Yury Person}

\thanks{%
  JH was supported by FAPESP (2014/18641-5, 2013/03447-6).
  YK was partially supported by FAPESP (2013/03447-6) and
  CNPq (310974/2013-5, 311412/2018-1, 423833/2018-9).
  PM is supported by a Leverhulme Trust Study Abroad
  Studentship (SAS-2017-052$\backslash$9).  YP is supported by the
  Carl Zeiss Foundation.  The cooperation of the authors was supported
  by a joint CAPES-DAAD PROBRAL project (Proj.\ no.~430/15, 57350402,
  57391197).
  FAPESP is the S\~ao Paulo Research Foundation.  CNPq is the National
  Council for Scientific and Technological Development of Brazil.%
}

\address{Department of Mathematics, University of Rhode Island, 5 Lippitt Road, Kingston, RI, 02881, USA}
\email{jie\_han@uri.edu}

\address{Instituto de Matem\'atica e Estat\'istica, Universidade de
  S\~ao Paulo, Rua do Mat\~ao 1010, 05508-090 S\~ao Paulo, Brazil}
\email{yoshi@ime.usp.br}

\address{Institut f\"ur Mathematik, Freie Universit\"at Berlin, Arnimallee 3, 14195 Berlin, Germany and Berlin Mathematical School, Germany}
\email{pm0041@mi.fu-berlin.de}

\address{Institut f\"ur Mathematik, Technische Universit\"at Ilmenau, 98684 Ilmenau, Germany}
\email{yury.person@tu-ilmenau.de}

\begin{document}
\shortdate
\yyyymmdddate
\settimeformat{ampmtime}
\footskip=28pt

\begin{abstract}
  We prove that for any $t\ge 3$ there exist constants $c>0$ and~$n_0$
  such that any $d$-regular $n$-vertex graph~$G$ with
  $t\mid n\geq n_0$ and second largest eigenvalue in absolute
  value~$\lambda$ satisfying $\lambda\le c d^{t}/n^{t-1}$ contains a
  $K_t$-factor, that is, vertex-disjoint copies of $K_t$ covering
  every vertex of~$G$. The result generalizes to broader setting of jumbled graphs, which were introduced by Thomason in the eighties.
\end{abstract}

\maketitle

\section{Introduction}
An $(n,d,\lambda)$-graph is an $n$-vertex $d$-regular graph whose
second largest eigenvalue in absolute value is at most~$\lambda$.
These graphs are central objects in extremal, random and algebraic
graph theory. The interest in these graphs lies in the fact that
various pseudorandom properties can be inferred from the value of
$\lambda$, in terms of the other parameters. For example, if
$\lambda \ll d$ then such a graph has the property that its edges are
`distributed' uniformly, which is one of the essential properties
exploited in random graphs and the regularity method from extremal
graph theory. More precisely, the following inequality, called the
\emph{expander mixing lemma} (see e.g.~\cite{AC88}), makes this
quantitative:
\begin{equation}
  \label{eq:EML.1}
  \left|e(A,B)-{d\over n}|A||B|\right|<\lambda\sqrt{|A||B|},
\end{equation}
whenever~$A$ and~$B$ are vertex subsets of an
$(n,d,\lambda)$-graph~$G$, where~$e(A,B)$ denotes the number of edges
between~$A$ and~$B$ (edges in $A\cap B$ are counted twice).  An
excellent introduction to the study of $(n,d,\lambda)$-graphs is given
in a survey of Krivelevich and Sudakov~\cite{KS06}.  The emphasis
there and throughout the field is on the interplay between the
parameters~$n$, $d$ and~$\lambda$ and graph properties of interest;
more precisely, given some property, one wishes to establish best
possible conditions on~$n$, $d$ and~$\lambda$ that ensure that any
$(n,d,\lambda)$-graph with parameters satisfying such conditions has
the property.

In this note we will only be concerned with conditions on the
parameters~$n$, $d$ and~$\lambda$ that guarantee the existence of
certain spanning structures, i.e., subgraphs that occupy the whole
vertex set (and have minimum degree at least one). Thus, we will be
somewhat selective in our discussion. In particular, we are interested
in whether, for some fixed~$t\geq 3$, our $(n,d,\lambda)$-graph
contains a family of vertex-disjoint copies of~$K_t$ covering each
vertex exactly once, which we call a
\textit{$K_t$-factor}.\footnote{This is also sometimes called a
  \textit{perfect $K_t$-matching} or a \textit{perfect $K_t$-tiling}
  in the literature.}  We remark that the case $d=\Theta(n)$ is
well-understood since the existence of bounded degree spanning graphs
in $(n,d,\lambda)$-graphs follows from the celebrated Blow-up lemma of
Koml\'os, S\'ark\"ozy and Szemer\'edi~\cite{KSS_bl}; see the
discussion in e.g.~\cite{HKP18a}.  Moreover, a sparse blow-up lemma
for subgraphs of $(n,d,\lambda)$-graphs~\cite{ABHKP16} gives general
nontrivial conditions for the existence of a given bounded degree
spanning subgraph in the case~$d=o(n)$.  These conditions\footnote{For reference, the main result in ~\cite{ABHKP16} requires $\lambda<c_3d^4n^{-3}$ for the existence of a $K_3$-factor and $\lambda<c_t\left(d/n\right)^{\frac{3(t-1)+1}{2}}n$ for~$t\geq 4$, where the $c_t=c(t)>0$ are appropriately chosen constants.} are stronger
than those discussed in what follows (and hence, as expected, the
general results in~\cite{ABHKP16} are weaker than the ones below). 

While extremal and random graph theory provide tools to answer
questions in this line of research positively, one is naturally
interested in the asymptotic tightness of the obtained results.  This
requires constructions of special pseudorandom graphs, most of the
known examples of which come from algebraic graph theory or geometry.
For our study here, there is essentially one prime example of such a
construction, due to Alon~\cite{Alon94}, who gave $K_3$-free
$(n,d,\lambda)$-graphs with $d=\Theta(n^{2/3})$ and
$\lambda=\Theta(n^{1/3})$.  Krivelevich, Sudakov and
Szab\'o~\cite{KSS04} then extended these to the whole possible range
of $d=d(n)$, constructing $K_3$-free $(n,d,\lambda)$-graphs with
$\lambda=\Theta(d^2/n)$ for all $\Omega(n^{2/3})\leq d \leq n$.  Alon's construction is an important
milestone in the study of $(n,d,\lambda)$-graphs.  It provides a rare
example of something that is reminiscent of threshold phenomena in
random graphs in the context of $(n,d,\lambda)$-graphs: if
$\lambda\le 0.1 d^2/n$, then any vertex of an $(n,d,\lambda)$-graph is
contained in a copy of $K_3$, while there are $K_3$-free
$(n,d,\lambda)$-graphs with $\lambda=\Omega(d^2/n)$.  Even more is
true: as proved by Krivelevich, Sudakov and Szab\'o~\cite{KSS04},
$(n,d,\lambda)$-graphs with $\lambda\le 0.1 d^2/n$ contain a
fractional triangle-factor.\footnote{A fractional $K_t$-factor in a
  graph $G$ is a function $f:\cK_t\to\mathbb{R}_+$, where $\cK_t$ is
  the set of copies of $K_t$ in $G$, such that
  $\sum_{v\in K\in \cK_t}f(K)=1$ for all $v\in V(G)$.}  A natural
conjecture from the same authors~\cite{KSS04} states the following.

\begin{conjecture}[Conjecture~7.1 in~\cite{KSS04}]
  \label{conj:KSS}
  There exists an absolute constant $c>0$ such that every
  $(n,d,\lambda)$-graph~$G$ on $n\in 3\mathbb{N}$ vertices with
  $\lambda\le cd^2/n$ has a triangle-factor.
\end{conjecture}

The first result in this direction was given by Krivelevich, Sudakov
and Szab\'o~\cite{KSS04}, who proved that it suffices to impose
$\lambda\leq cd^3/(n^2\log n)$ for some absolute constant~$c>0$.  This
was improved by Allen, B\"ottcher, H\`an and two of the current
authors~\cite{ABHKP17} to $\lambda\leq cd^{5/2}/n^{3/2}$ for
some~$c>0$.  In fact, the result in~\cite{ABHKP17} is that this
condition on~$\lambda$ is enough to guarantee the appearance of
squares of Hamilton cycles (the square of a Hamilton cycle is obtained
by connecting distinct vertices at distance at most~$2$ in the cycle).
Another piece of evidence in support of Conjecture~\ref{conj:KSS} is a
result in~\cite{HKP18a, HKP18b} that states that, under the condition
$\lambda\le (1/600) d^2/n$, any $(n,d,\lambda)$-graph~$G$ with~$n$
sufficiently large contains a `near-perfect $K_3$-factor'; in fact,
$G$~contains a family of vertex-disjoint copies of $K_3$ covering all
but at most $n^{647/648}$ vertices of~$G$.  Very recently,
Nenadov~\cite{Nen18} proved that $\lambda\leq cd^2/(n\log n)$ for some
constant~$c>0$ yields a $K_3$-factor.\footnote{
Nenadov studied a larger class of graphs, which contains
  $(n,d,\lambda)$-graphs as a special case.  Our result also holds in this class. For details, see the concluding   remarks at the end of this note.
}

With regard to $K_t$-factors for general $t\in \mathbb{N}$,
Krivelevich, Sudakov and Szab\'o~\cite{KSS04} remark that the
condition $\lambda\le cd^{t-1}/n^{t-2}$ yields, for appropriate
$c=c(t)$, a fractional $K_t$-factor.  Although there is, alas, no
known suitable generalization of Alon's construction to $K_t$-free
graphs, this condition on~$\lambda$ may be seen as a benchmark in the
study of $K_t$-factors in $(n,d,\lambda)$-graphs.  The first
nontrivial result for this study was the result in~\cite{ABHKP17}
showing that $\lambda\leq cd^{3t/2}n^{1-3t/2}$ for some~$c=c(t)>0$ is
sufficient to guarantee $t$-powers of Hamilton cycles (and thus
$K_{t+1}$-factors when $(t+1)\mid n$).  The aforementioned result of
Nenadov~\cite{Nen18} generalizes to $K_t$-factors, giving
the condition $\lambda\leq cd^{t-1}/(n^{t-2}\log n)$ for some
constant~$c=c(t)>0$.  The purpose of this note is to present a proof
that, under the condition $\lambda\leq cd^t/n^{t-1}$ for some
suitable~$c=c(t)>0$, any $(n,d,\lambda)$-graph contains a
$K_t$-factor.  More precisely, we prove the following.

\begin{theorem}
  \label{thm:main}
  Given an integer $t\ge 3$, there exist $c>0$ and $n_0>0$ such
  that every $(n, d, \lambda)$-graph with $n\ge n_0$ and
  $\lambda \le c d^t/n^{t-1}$ contains a $K_t$-factor.
\end{theorem}

If~$d\geq cn/\log n$ for some suitable~$c=c(t)>0$ and~$t\geq4$, then
the condition in Theorem~\ref{thm:main} is the weakest that is
currently known to imply the existence of $K_t$-factors.
Theorem~\ref{thm:main}, first announced in~\cite{HKP18a}, was obtained
independently of Nenadov's result.  Although both approaches 
adopt absorption techniques, there are few similarities in the key ideas of the two proofs. Indeed, both arguments work by finding a suitably defined absorbing structure which can complete a $K_t$-factor in many ways but the two structures used are clearly distinct and the heart of each proof lies in defining and proving the existence of the desired structure in the host graph. Intuitively, it seems the absorbing structure used by Nenadov~\cite{Nen18} is sparser than ours which explains why it is more effective in the sparse range.   
 
We use standard notation from graph theory, see e.g.~\cite{West}. We
will omit floor and ceiling signs in order not to clutter the
arguments.

\section{Tools}

\subsection{Properties of $(n,d,\lambda)$-graphs}
We begin by giving some basic properties of $(n,d,\lambda)$-graphs.
Some of these are well known and used throughout the study of
$(n,d,\lambda)$-graphs whilst others are specifically catered to our
purposes here.

\begin{theorem}[Expander mixing lemma~\cite{AC88}]
  \label{thm:EML}
  If $G$ is an $(n,d,\lambda)$-graph and $A$, $B\subseteq V(G)$, then
  \begin{equation}
    \label{eq:EML}
    \left|e(A,B)-\frac dn|A||B|\right|<\lambda\sqrt{|A||B|}.    
  \end{equation}
\end{theorem}

\begin{proposition}[Proposition~2.3
  in~\cite{KSS04}]\label{prop:lambda}
  Let $G$ be an $(n,d,\lambda)$-graph with~$d\le n/2$.  Then
  $\lambda\ge \sqrt{d/2}$.
\end{proposition}

\begin{fact}
  \label{fact:d_bound}
  Let $G$ be an $(n,d,\lambda)$-graph with $d\le n/2$.  Suppose
  $\lambda\le d^t/n^{t-1}$ for some $t\ge2$.  Then
  $d\ge n^{1-1/(2t-1)}/2$.
\end{fact}
\begin{proof}
  Proposition~\ref{prop:lambda} tells us that $\lambda\ge \sqrt{d/2}$.
  Thus $\lambda\le d^t/n^{t-1}$ implies that
  $d^{2t-1}\ge n^{2t-2}/2$, whence
  $d\ge n^{1-1/(2t-1)}/2^{1/{(2t-1)}}$ follows.
\end{proof}

\begin{fact}\label{fact:irregular}
  Let $G$ be an $(n,d,\lambda)$-graph with $d\leq n/2$ and
  $\lambda \le \eps d^{t}/n^{t-1}$. If $U$ is a set of
  $m'\ge d/2$ vertices, then there are at most $\eps d$ vertices $u$
  with $|N_G(u)\cap U| < d m'/(2n)$.
\end{fact}
\begin{proof}
  Let $U'$ be the set of vertices $u$ such that
  $|N_G(u)\cap U| < d m'/(2n)$. By Theorem~\ref{thm:EML}, we have
  \[
    \frac dn m' |U'| - \lambda \sqrt{ m' |U'| } < e(U, U') < |U'|
    \frac{d m'}{2n}.
  \]
  Together with $\lambda \le \eps d^{t}/n^{t-1}$, we obtain that
  $|U'|\le 8\eps^2 d^{2t-3}/n^{2t-4}\le \eps d$.
\end{proof}

Write $\Ktm$ for the graph obtained by replacing one vertex of~$K_t$
by an independent set of size~$40$.
  
\begin{fact}
  \label{fact:cliquecount}
  Let $G$ be an $(n, d, \lambda)$-graph with
  $\lambda \le cd^{t}/n^{t-1}$ and suppose
  $cd^{t}/n^{t-1}\ge40$.  Then any set of $2^tcd$
  vertices spans a copy of~$\Ktm$.  Moreover, any set of
  $2^{t-1}cd^2/n$ vertices spans a copy of~$K_{t-1}$.
\end{fact}
\begin{proof}
  Let $U$ be a set of at least $2cd^{t-1}/n^{t-2}$ vertices in~$G$.
  Since $\lambda \le cd^{t}/n^{t-1}\le d|U|/(2n)$, it follows
  from Theorem~\ref{thm:EML} that
  \[
    2e(U)\ge\frac dn|U|^2 -\lambda|U|\ge|U|\cdot \frac{d|U|}{2n},
  \]
  which implies that~$U$ contains a vertex with degree at
  least~$d|U|/(2n)$.

  Thus, given a set of $2^tcd$ vertices, we can iteratively pick
  vertices with large degree in the common neighbourhood, and get a
  $(t-1)$-clique whose common neighbourhood has size at least
  $2^tcd\cdot(d/(2n))^{t-1} = 2cd^{t}/n^{t-1} \ge 40$ (the
  smallest set from which we pick a vertex in this process has size
  $4cd^{t-1}/n^{t-2}$).  Therefore, we obtain a copy of~$\Ktm$.  The
  proof of the `moreover'-part is analogous.
\end{proof}

\subsection{Templates and absorbing structures}

A \emph{template} $T$ with \emph{flexibility} $m\in \mathbb{N}$ is a
bipartite graph on $7m$ vertices with vertex parts $I$ and
$J:=J_1\dcup J_2$, such that $|I|=3m$, $J_1=[2m]$, $J_2=[2m+1,4m]:=\{2m+1,\ldots,4m\}$, and for any
$\bar{J}\subset J_1$, with $|\bar{J}|=m$, the induced graph
$T[V(T)\setminus\bar{J}]$ has a perfect matching. We call $J_1$ the
\textit{flexible} set of vertices for the template. Montgomery first
introduced the use of such templates~\cite{M14a} (see also~\cite{Kwan16,FN17}).
There, he used a sparse template of maximum degree~$40$, which we will
also use.  It is not difficult to prove the existence of such
templates for large enough~$m$ probabilistically; see
e.g.~\cite[Lemma~2.8]{M14a}. 

Here, we will use a sparse template to build an absorbing structure
suitable for our purposes. The absorbing structure we will use is
defined as follows.  Let $m$ be a sufficiently large integer.  Let
$T=(I, J_1\cup J_2, E)$ be a template with flexibility
$m$, maximum degree $\Delta(T)\leq 40$ and flexible set $J_1$.  Write
 $Z_1=\{z_1,\dots, z_{2m}\}$,
$Z_2=\{z_{2m+1}, \dots, z_{4m}\}$ and define $Z:=Z_1\cup Z_2$. An
\emph{absorbing structure} $(T, \cK, A, \cS, Z, Z_1)$ with
flexibility~$m$ contains two sets $\cK$ and $\cS$ consisting of
vertex-disjoint $(t-1)$-cliques and two vertex sets $A$ and $Z$ such
that $V(\cK)$, $V(\cS)$, $A$ and $Z$ are pairwise disjoint, $|\cK|=3m$ and $|A|=|\cS|=|E(T)|$.
Furthermore, with the labelling $\cK:=\{K^i\colon i\in I\}$,
$A=\{a_{ij}: ij\in E(T)\}$ and
$\cS=\{S_{ij}: ij\in E(T)\}$, the following holds.  For all $i\in I$
and~$j\in J$ such that $ij\in E(T)$,
\begin{itemize}
\item each $\{a_{ij}\}\cup K^{i}$ spans a copy of $K_t$,
\item each $\{a_{ij}\}\cup S_{ij}$ spans a copy of $K_t$, and
\item each $\{z_{j}\}\cup S_{ij}$ spans a copy of $K_t$.
\end{itemize}
We also call the set $Z_1$ flexible since it corresponds to the flexible set $J_1$ of $T$. 
\begin{fact}
  An absorbing structure $(T, \cK, A, \cS, Z, Z_1)$ with flexibility $m$ defines a graph~$H$ on a set of~$4m+3m(t-1)+e(T)t\leq 3mt(\Delta(T)+2)$ vertices which has the following property.
  For any subset $\bar Z\subseteq Z_1$ with $|\bar Z|=m$, the
  graph induced by $H$ on $V(H)\setminus \bar Z$ has  a $K_t$-factor.
\end{fact}
\begin{proof}
  The graph $H$ is obtained by taking all the vertices and edges which feature in the definition of the absorbing structure. The vertex count is then clear and it remains to establish the existence of a $K_t$-factor when some arbitrary $\bar Z$ is removed. By the property of the template $T=(I, J_1\dcup J_2, E)$, there is a
  perfect matching~$M$ in~$T$ that covers~$I$ and $J\setminus \{j\colon z_j\in \bar{Z}\}$.
  Then for each edge $ij\in M$, we take the $t$-cliques on
  $\{a_{ij}\}\cup K^{i}$ and $\{z_{j}\}\cup S_{ij}$; for the
 edges $ij\in E(T)\setminus M$, we take the
  $t$-cliques on $\{a_{ij}\}\cup S_{ij}$.  This gives the desired
  $K_t$-factor.
\end{proof}
 
The following lemma asserts that $(n,d,\lambda)$-graphs possess
the absorbing structures above.

\begin{lemma}
  \label{lem:1}
  There exists $\eps_0>0$ such that for all $\eps\in (0,\eps_0)$ there
  is an $n_0\in \mathbb{N}$ such that the following holds for all
  $n\ge n_0$.  Let~$G$ be an $(n, d, \lambda)$-graph with
  $\lambda \le \eps d^{t}/n^{t-1}$ and suppose $m=\eps d$.  Then there
  exists an absorbing structure $(T, \cK, A, \cS, Z, Z_1)$ with
  flexibility~$m$ such that, for any vertex~$v$ in~$G$, we have
  $\deg(v, Z_1) \ge d|Z_1|/(2n)$.
\end{lemma}

The following concentration result will be used in the proof of Lemma~\ref{lem:1}

\begin{lemma}[Lemma~2.2 in~\cite{ABHKP16}]
  \label{lem:coupling} Let $\Omega$ be a finite probability space and
  let $\cF_0\subseteq \dots\subseteq \cF_n$ be partitions of $\Omega$ such that $\cF_{i}$ is a refinement of $\cF_{i-1}$ for each $i\in[n]$.  Further, for each $i\in[n]$ let $Y_i$ be a Bernoulli random
  variable on $\Omega$ that is constant on each part of~$\cF_i$ and let~$p_i$ be a
  real-valued  random variable on~$\Omega$ which is constant on each part of $\cF_{i-1}$.
  Let~$x$ and~$\delta$ be  real numbers with $\delta\in(0,3/2)$, and
  let $X=Y_1+\dots+Y_n$.  If $\sum_{i=1}^n p_i\ge x$ holds almost
  surely and $\EE [Y_i\mid\cF_{i-1}]\ge p_i$ holds almost surely for
  all $i\in[n]$, then
  \[\Pr\big(X<(1-\delta)x \big)<e^{-\delta^2 x/3}\,.\]
\end{lemma}

\begin{proof}[Proof of Lemma~\ref{lem:1}]
  First we choose $\eps_0=1/(300t)$ and let $\eps\in(0,\eps_0)$. Then
  we take $n_0$ large enough.

  Let $T=(I, J_1\cup J_2, E)$ be a bipartite template with
  flexibility~$m$ and flexible set $J_1$ such that $\Delta(T)\le 40$.
  Pick an arbitrary collection of~$3m$ vertex-disjoint copies of
  $\Ktm$ (using Fact~\ref{fact:cliquecount}).  For each~$i\in I$, consider the $i^{th}$ copy of $\Ktm$ and label
  the copy $K_{t-1}$ inside this $\Ktm$ as $K^i$. Further, consider the set of $40$ vertices which lie in the joint neighbourhood of the vertices of $K^i$ and label this set with the labels
  $A=\{a_{ij}: ij\in E(T)\}$ 
  (if the degree of~$i$ in~$T$ is less than $40$ we simply discard any excess vertices which lie in this copy of $\Ktm$). By design, we have the property that each $a_{ij}$
  together with $K^i$ forms a copy of $K_t$. 

  We will pick $Z=\{z_1,\dots, z_{4m}\}$ and
  $\cS=\{S_{ij}: ij\in E(T)\}$ satisfying the definition of the
  absorbing structure as follows.  Suppose that we have picked
  $Z(j-1)=\{z_1,\dots, z_{j-1}\}$ and $\cS(j-1)=\{S_{ij'}: j'<j\}$
  with the desired properties.  At step~$j$, we pick as~$z_{j}$ a
  uniform random vertex in
  $V(G)\setminus (V(\cK)\cup Z(j-1)\cup V(\cS(j-1))\cup B_j)$, where
  $B_j$ is the set of vertices $z$ in $G$ such that
  \[\left|\Big(N_G(a_{ij})\setminus \big(V(\cK)\cup Z(j-1)\cup V(\cS(j-1))\big)\Big)\cap
  N_G(z)\right| < d^2/(4n),\] for some $i$ with $ij\in E(T)$.  Since
  $|V(\cK)| + 4m + 120(t-1)m\le (123t+1) \eps d<d/2$ and
  $\Delta(T)\le 40$, Fact~\ref{fact:irregular} with
  $U=N_G(a_{ij})\setminus (V(\cK)\cup Z(j-1)\cup V(\cS(j-1)))$ implies
  that $|B_{j}|\le 40\eps d$.  Next, for each $i$ such that
  $ij\in E(T)$, we pick a $(t-1)$-clique $S_{ij}$ in
  \[\Big(N_G(a_{ij})\setminus \big(V(\cK)\cup Z(j-1)\cup V(\cS(j-1))\big)\Big)\cap
  N_G({z_j}),\] which is possible by Fact~\ref{fact:cliquecount}
  because this set contains at least $d^2/(4n)$ vertices of $G$.
  Moreover, we can choose these at most $40$ cliques to be
  vertex-disjoint, because they only take up $40(t-1)$ vertices and
  any set in $G$ of size $d^2/(4n)-40(t-1) > d^2/(5n)$ still contains
  a $(t-1)$-clique.

  At last we analyse the random process for~$Z$ and prove that, with
  positive probability, all vertices~$v$ of~$G$ are such
  that~$\deg(v,Z_1)$ is appropriately large.  Note that at step $j$ we
  have fixed
  $|V(\cK)\cup V(\cS(j-1)) \cup Z(j-1)|\le {3}(t+39)m+120(t-1)m+4m \le
  (123t+1) \eps d$ vertices.  Since when we pick $z_j$, we also need
  to avoid the set $B_j$ of size at most $40\eps d$, in total we need
  to avoid at most $140t\eps d$ vertices.  Let $v$ be a vertex in
  $G$. Given a choice of $V(\cK)\cup V(\cS(j-1)) \cup Z(j-1)$, the
  probability that $z_j\in N_G(v)$ is at least $(1-{140t}\eps) d/n$.
  Then, by Lemma~\ref{lem:coupling} with $\delta=\eps$, we have
  \[
    \Pr\left(\deg(v, Z_1) < \frac d{2n}
      |Z_1|\right)<\Pr\left(\deg_G(v, Z_1) < (1- \eps)
      (1-{140t}\eps)|Z_1|d/n\right) < e^{-\eps^3 d^2/n} =
    o\left(1\over n\right).
  \]
  Thus, the union bound over all vertices of~$G$ implies that the
  existence of~$Z_1$ with the desired property in the lemma.
\end{proof}

\subsection{A Hall-type result}
Another tool that we will use is the following theorem of Aharoni and
Haxell~\cite[Corollary~1.2]{AhHa00}.

\begin{theorem}
  \label{AhHa}
  Let $\cH$ be a family of $k$-uniform hypergraphs {on the same vertex
    set}.  A sufficient condition for the existence of a system of
  disjoint representatives\footnote{By this we mean a selection of
      edges $e_H\in H$ for all~$H\in\cH$ such that
      $e_H\cap e_H'=\emptyset$ for all $H\neq H'\in \cH$.} for $\cH$
  is that for every $\cG\subseteq \cH$ there exists a matching in
  {$\bigcup_{H\in \cG}E(H)$} of size greater than $k(|\cG|-1)$.
\end{theorem}

\section{Proof of Theorem~\ref{thm:main}}
We are now ready to prove our main result.
\begin{proof}[Proof of Theorem~\ref{thm:main}]
  Let~$t\geq3$ be given.  Let $\eps_0$ be given by Lemma~\ref{lem:1}.
  Choose $\eps:=\min\{\eps_0,t^{-2}\}$ and let~$n_0$ be given by
  Lemma~\ref{lem:1} on input $\eps_0$.  Finally, set
  $c=\eps^2/2^{t+1}$.  We may assume that
  \begin{equation}\label{eq:blow_up}
    {d\over n}\le{\eps\over2t},
  \end{equation}
  since otherwise the existence of a $K_t$-factor is guaranteed by the
  Blow-up Lemma~\cite{KSS_bl} since then the host graph has linear degree (see the discussion in~\cite{HKP18a}).
  We apply Lemma~\ref{lem:1} to~$G$ and obtain an absorbing structure
  $(T, \cK, A, \cS, Z, Z_1)$ with flexibility~$m=\epsilon d$ on a
  set~$W$ of at most $126t\eps d$ vertices.  Thus $Z_1\subset W$ is
  such that $|Z_1|=2\eps d$ and, for any subset $\bar Z\subseteq Z_1$
  with $|\bar Z|=\eps d$, the absorbing structure with $\bar Z$
  removed has a $K_t$-factor.  Moreover, $\deg(v, Z_1) \ge d|Z_1|/2n$
  for any vertex~$v$ in~$G$.

  Now we greedily find vertex-disjoint copies of~$K_t$ in
  $G\setminus W$ as long as there are at least $\eps^2 d$ vertices
  left. This is possible by Fact~\ref{fact:cliquecount} because
  $\eps^2 d > 2^tcd$.  We denote the set of uncovered vertices in
  $V(G)\setminus W$ by~$U$.  Thus $|U|\le \eps^2 d$.

  Next we will cover $U$ by vertex-disjoint copies of $K_t$ with one
  vertex in $U$ and the other vertices from $Z_1$ by applying
  Theorem~\ref{AhHa}.  To that end, for each vertex $v\in U$, let
  $H_v$ be the set of $(t-1)$-element sets of $N(v)\cap Z_1$ that
  induce copies of $K_{t-1}$ in $G$ and let $\cH=\{H_v: v\in U\}$.  We
  claim that~$\cH$ has a system of disjoint representatives.  To
  verify the assumption of Theorem~\ref{AhHa}, we first consider sets
  $X\subseteq U$ of size at least $d^{2t-3}/n^{2t-4}$.  Let $Z'$ be
  any subset of $Z_1$ of size~$\eps d$.  Note that
  $|X||Z'|\ge \eps d^{2t-2}/n^{2t-4}$, which implies that
  $\lambda \le cd^{t}/n^{t-1}\le \eps(d/n)\sqrt{|X||Z'|}$.  By
  Theorem~\ref{thm:EML}, we have
  \[
    e(X, Z')\ge \frac dn |X| |Z'| - \lambda \sqrt{|X||Z'|} \ge \frac
    dn |X| |Z'| - \eps \frac dn |X| |Z'| \ge \frac d{2n} |X| |Z'|.
  \]
  Hence there exists a vertex $v\in X$ such that
  $\deg(v, Z')\ge d|Z'|/(2n) = \eps d^2/(2n)$.  By
  Fact~\ref{fact:cliquecount}, we can find a copy of $K_{t-1}$ in
  $N(v)\cap Z'$.  Thus in this case we can greedily find
  vertex-disjoint copies of $K_{t-1}$ which belong to
  $\bigcup_{v\in X} H_v$, as long as there are $\eps d$ vertices in
  $Z_1$ left uncovered.  This gives a matching in
  $\bigcup_{v\in X} H_v$ of size
  $\eps d/(t-1) > (t-1)\eps^2 d \ge (t-1)|X|$, for $\eps$ sufficiently
  small.  It remains to consider sets $X\subseteq U$ of size at most
  $d^{2t-3}/n^{2t-4}$.  In this case we fix any vertex $v\in X$ and
  only consider matchings in~$H_v$.  Indeed, by the construction of
  the absorbing structure (see Lemma~\ref{lem:1}), we have
  $\deg(v, Z_1) \ge d|Z_1|/2n = \eps d^2/n$ for any~$v$.  Since by
  Fact~\ref{fact:cliquecount} there is a copy of $K_{t-1}$ in every
  set of size $\eps d^2/(2n)$, we can find a set of
  $\eps d^2/(2(t-1)n)$ vertex-disjoint copies of~$K_{t-1}$ in~$H_v$.
  We are done because $\eps d^2/(2(t-1)n) \ge d^{2t-3}/n^{2t-4}$ holds
  by our initial assumption~\eqref{eq:blow_up}.  Theorem~\ref{AhHa}
  then tells us that a system of disjoint representatives does exist
  for~$\cH$, whence we conclude that there are vertex-disjoint copies
  of $K_t$ which cover all the vertices in $V(G)\setminus W$ and
  $(t-1)\eps^2d$ vertices of~$Z_1$.  
  
  We can then greedily find
  vertex-disjoint copies of~$K_t$ in the remainder of $Z_1$, which
  exist by Fact~\ref{fact:cliquecount}, until exactly $\eps d$
  vertices of $Z_1$ remain (which will be the case due to the
  divisibility assumption $t\mid n$). Then the key property of the
  absorbing structure completes a full $K_t$-factor.
\end{proof}

\section{Concluding remarks}\label{sec:remark}
\subsection*{Jumbled graphs} 
The study of pseudorandom graphs was initiated by Thomason in~\cite{Th87a,Th87b} where he began to explore and prove properties of such graphs using the following definition. 
A graph $G$ on $n$ vertices is called \emph{$(\lambda,p)$-jumbled} ($0<p\le 1\le \lambda$) if for
any vertex subset $U\subseteq V(G)$, 
\begin{equation*}
    \left|e(U)-p\binom{|U|}{2}\right|\le\lambda |U|    
  \end{equation*}
holds. Thus, by Theorem~\ref{thm:EML},
any $(n,d,\lambda)$-graph is $(\lambda,d/n)$-jumbled. On the other hand a $(\lambda,p)$-jumbled graph has average degree roughly $pn$ and all but $O(\lambda/p)$ vertices  have degree $\Omega(pn)$.

Theorem~\ref{thm:main} can be easily adapted to $(\eps_t p^tn,p)$-jumbled graphs $G$ of minimum degree $\Omega(pn)$ for appropriate constant $\eps_t$ as follows. Since at most $O(\lambda/p)$ vertices have few neighbours into a set of size $\Omega(pn)$, we can recover Fact~\ref{fact:irregular} with $d$ replaced by $pn$ and slightly altered constants. Fact~\ref{fact:cliquecount} about counting cliques can also be carried over to $(\eps_t p^tn,p)$-jumbled graphs almost verbatim and hence Lemma~\ref{lem:1} as well (where one also needs the fact that $\lambda=\Omega(\sqrt{pn})$, see e.g.~\cite[page~6]{KS06}). The only other adjustment is needed in the proof of Theorem~\ref{thm:main} where we need to verify a Hall-type condition (Theorem~\ref{AhHa}) by considering subsets $X$ of size more than $\eps p^2n/(2(t-1))$ and subsets $X$ of size at most $\eps p^2n/(2(t-1))$. Again, one can proceed almost verbatim as in the proof of Theorem~\ref{thm:main}, replacing only the numerical values appropriately. This yields the following result.
\begin{theorem}
  Given an integer $t\ge 3$ and a real $c\in (0,1]$, there exist $\eps_t>0$ and $n_0>0$ such
  that every $(\eps_t p^t n,p)$-jumbled graph~$G$ with $t\mid n$,  $n\ge n_0$ and $\delta(G)\ge cpn$, 
 contains a $K_t$-factor.
\end{theorem}

A slight variation allowing more flexibility when working with sets of different sizes is the notion of \emph{bijumbled graphs} introduced in~\cite{KRSSS07}. A graph $G$ on $n$ vertices is called $(\lambda,p)$-bijumbled if for
any two vertex subsets $A$, $B\subset V(G)$, 
\begin{equation*}
    \left|e(A,B)-p|A||B|\right|<\lambda\sqrt{|A||B|}.    
  \end{equation*}
Clearly, any $(n,d,\lambda)$-graph is $(\lambda,d/n)$-bijumbled and any $(\lambda,p)$-bijumbled graph is $(\lambda,p)$-jumbled. 

Nenadov's result~\cite{Nen18} asserts that for $p\in(0,1]$, any 
$(\epsilon p^{t-1}n/\log n,p)$-bijumbled graph of minimum degree at
least~$pn/2$ contains a $K_t$-factor if~$\epsilon=\epsilon_t>0$ is
sufficiently small and~$t\mid n$.  

\subsection*{A condition for arbitrary $2$-factors}
In his concluding remarks, Nenadov~\cite{Nen18} raises the question
whether the condition $\lambda=o(p^2n/\log n)$ is sufficient to force
any $(\lambda,p)$-bijumbled graph~$G$ of minimum degree $\Omega(pn)$
to contain any given $2$-factor, i.e., any $n$-vertex $2$-regular
graph.  Since any $2$-factor consists of vertex-disjoint cycles whose
lengths add up to~$n$, the problem is thus to find any given
collection of such cycles in $G$.  We will return to this question
elsewhere~\cite{HKMP18b}, with a positive answer to Nenadov's
question.

\bibliography{literature}

\end{document}